\documentclass[smallextended]{svjour3}

\usepackage{amsmath,amssymb,bbm}
\usepackage{color,comment,xfrac}

\newcommand{\Tr}{\operatorname{Tr}}
\newcommand{\supp}{\operatorname{supp}}
\newcommand{\HS}{\mathrm{HS}}

\newcommand{\mydot}{{\,\cdot\,}}
\newcommand{\fH}{\mathcal H}
\newcommand{\fK}{\mathcal K}
\newcommand{\mhh}{(-1/2,1/2)}
\newcommand{\bbC}{\mathbb{C}}
\newcommand{\bbR}{\mathbb{R}}
\newcommand{\bbZ}{\mathbb{Z}}
\newcommand{\bbN}{\mathbb{N}}
\newcommand{\ths}{\text{HS}}
\newcommand{\htil}{\tilde{\mathcal H}}
\newcommand{\ltht}{L^2(\htil)}
\newcommand{\bone}{\ensuremath{\mathbbm 1}}
\newcommand{\gbg}{\Gamma\backslash G}

\newcommand{\myphi}{\phi}

\newcommand{\rhonormk}{\left\Vert \rho_{\omega_k}(f)\right\Vert _{\HS}^{2}}
\newcommand{\hn}{n} % make empty in one-dimensional case
\newcommand{\mycdot}{\cdot} % make empty in one-dimensional case
\newcommand{\bbH}{ H_{\hn}}
\newcommand{\myphikf}{|\omega_k|^{\hn} \rhonormk}
\newcommand{\mqqo}{(-1/4,1/4)}
\newcommand{\dsintparams}{{ \bbZ^{2\hn}\times2\bbZ}}

\begin{document}

\title{Operator-Valued Frames for the Heisenberg Group\thanks{This work was supported by the U.S.~Air Force Office of Scientific Research under Grant No.~FA9550-12-1-0418.}
}

\author{Benjamin Robinson \and William Moran
        \and Douglas Cochran \and Stephen D.~Howard}
\authorrunning{Robinson et al.} 

\institute{Benjamin Robinson and Douglas Cochran\at
              School of Mathematical \& Statistical Sciences \\
              Arizona State University \\
              Tempe, AZ 85287-1809 USA
              \email{Benjamin.Robinson@asu.edu, cochran@asu.edu}  
 \and
          William Moran \at
          RMIT University \\
          Melbourne, VIC  3000 Australia
          \email{Bill.Moran@rmit.edu.au}
 \and
          Stephen D.~Howard \at
          Defence Science \& Technology Organisation \\
          Edinburgh, SA  5111 Australia
          \email{stephen.howard@dsto.defence.gov.au}
}


\maketitle
\begin{abstract}
A classical result of Duffin and Schaeffer gives conditions under which a discrete collection of characters on $\bbR$, restricted to $E=\mhh$, forms a Hilbert-space frame for $L^2(E)$. For the case of characters with period one, this is just the Poisson Summation Formula. Duffin and Schaeffer show that perturbations preserve the frame condition in this case.  This paper gives analogous results for the real Heisenberg group $H_n$, where frames are replaced by \emph{operator-valued frames}. The Selberg Trace Formula is used to show that perturbations of the orthogonal case continue to behave as operator-valued frames.  This technique enables the construction of decompositions of elements of $L^2(E)$ for suitable subsets $E$ of $H_n$ in terms of representations of $H_n$. 

\vspace*{0.2cm}\noindent
{\bf Keywords} Operator-valued frames, G-frames, Representations, Heisenberg Group, Sampling

\vspace*{0.2cm}\noindent
{\bf AMS Subject Classification} 42C15
\end{abstract}

%%%%%%%%%%%%%%%%%%%%%%%%%%%%%%%%%%%%%%%%%%%%%%%%%%%%%%%%%%%%%%%%%%%%%%%%%%%%%%%%%%
\section{Introduction}

A frame for a Hilbert space $\fH$ is a sequence $\{\xi_j\, |\, j\in\bbN\} \subset \fH$ such that
\[
A\left\Vert \xi\right\Vert _{\fH}^2\le\sum_{j\in\bbZ} \left|\left< \xi_j,\xi\right>_{\fH} \right|^2\le B\left\Vert \xi\right\Vert _{\fH}^2
\]
for positive real numbers $A$ and $B$ and all $\xi\in \fH$.  In the case that $\{\xi_j\}$ is an orthonormal basis for $\fH$, this sequence is a Parseval frame for $\fH$, meaning $A=B=1$.  In particular, when $\fH=L^2(-1/2,1/2)$, the sequence $\{\xi_j\}$ may be taken to be the standard Fourier basis $\{e^{2\pi i j \mydot}\, |\, j\in \bbZ\}$. In their seminal 1952 paper \cite{duffin1952class}, Duffin and Schaeffer proved that it is possible to perturb the numbers $j\in\bbZ$ and still preserve the frame condition:
\begin{theorem}
  \label{dsresult}
Denote by $E$ the interval $\mhh\subset\bbR$.  Let $M>0$ and $\delta>0$, and suppose $\{\chi_j=e^{2\pi i \omega_j\cdot}\, |\, j\in\bbZ\}$ is a sequence of characters of $\bbR$ with
\begin{enumerate}
\item $|\omega_{j}-j|<M$, for all $j\in \bbZ$, and
\item $|\omega_j-\omega_k|\geq \delta$ for all $j\neq k\in \bbZ$.
\end{enumerate}
Then there exist positive real numbers $B\geq A>0$ such that, for any $f\in L^2(E)$,
\begin{equation*}
A\left\Vert f\right\Vert _{L^2(E)}^2\le\sum_{j\in\bbZ} \left|\left< \chi_j,f\right>_{L^2(E)} \right|^2\le B\left\Vert f\right\Vert _{L^2(E)}^2 .
\end{equation*}
\end{theorem}
Modern literature in frame theory widely acknowledges that the subject was initiated by this paper of Duffin and Schaeffer, although it received relatively little attention until the late 1980s when Daubechies, Grossmann, and Meyer connected this important idea with the rapidly expanding area of wavelet analysis \cite{daubechies1986painless}.  Theorem \ref{dsresult} has inspired numerous generalizations and extensions, including the highly significant work of Kadets (Kadec) \cite{kadec1964exact} and Avdonin \cite{avdonin1974} on Riesz bases of exponentials. These and other results are described in \cite{lindner1999}. 

The objective of this paper is to extend these perturbation results to the real Heisenberg group $\bbH$ in a natural way. In this context, the appropriate notion of a frame is a ``g-frame'' or ``operator-valued frame.'' This concept is found in \cite{sun2006g} and also appears in \cite{casazza2004frames,casazza2008fusion,kaftal2009operator,Han:2014kx}. In what follows, the term \emph{operator-valued frame} (OVF) will be used to mean a  countable sequence of (arbitrary-rank) linear operators $T_j:\fH\rightarrow \fK_j$ mapping a separable, complex Hilbert space $\fH$ into separable, complex Hilbert spaces $\fK_{j}$ with the property that there are positive real numbers $B \ge A > 0$ such that
\[
A\left\Vert \xi\right\Vert^2\le\sum_j\left\Vert T_j \xi\right\Vert _{\fK_j}^2\le B\left\Vert \xi\right\Vert^2 
\]
for all $\xi\in \fH$. In some papers, the definition includes the additional assumption that the maps $T_j$ all have the same rank, but that condition will not be required. A frame in the usual sense is a special case of this definition with the additional requirement that $\dim \fK_j=1$ for all $j$. Much of the study of OVFs thus far has been motivated by multiwavelets \cite{kaftal2009operator} and by certain applications in distributed processing \cite{casazza2008fusion}. 

One straightforward way to construct OVFs \cite{kaftal2009operator,Han:2014kx} is to fix a locally compact group $G$, a unitary representation $\pi$ of $G$ on
a Hilbert space $\fH$, and a sequence of points $\{x_j\}\subset G$, and set $T_j=T\pi(x_j)$ where $T$ is some fixed operator from $\fH$ into a separable, complex Hilbert space $\fK$. This recipe is also used to generate many rank-one frames of interest \cite{daubechies1986painless}. However, as Theorem \ref{dsresult} suggests, there are other constructions for frames.  Among them are the  so-called \emph{frames of exponentials}, as described in Theorem~\ref{dsresult}. Others have extended Theorem~\ref{dsresult} to more general locally compact abelian groups (see \cite{feichtinger1992irregular}).  

The general goal of this paper is to extend the circle of results of Duffin and Schaeffer to the setting of a non-trivial subspace $\fH$ of $L^2(G)$ where $G$ is a unimodular locally compact Lie group. In the case of a noncommutative group, such as $H_n$, there will of course be representations of dimension greater than one. This means that frames will have to be replaced by OVFs. 

Some conditions are presented in Section~\ref{sec:Harmonic-g-frames-of-reps} for $\{\pi_j\}$ corresponding, in a fairly general context, to the orthogonality of harmonic frames of exponentials in the classical case studied by Duffin and Schaeffer. Section~\ref{sec:heis} sets forth results analogous to Theorem~\ref{dsresult} for the special case where $G$ is $\bbH$.

%%%%%%%%%%%%%%%%%%%%%%%%%%%%%%%%%%%%%%%%%%%%%%%%%%%%%%%%%%%%%%%%%%%%%%%%%%%%%%%%%%
\section{Harmonic OVFs of representations}
\label{sec:Harmonic-g-frames-of-reps}

This section begins with a synopsis of a few aspects of representation theory needed in subsequent discussion. It proceeds to describe some Parseval OVFs. The setting is a unimodular locally compact Lie group $G$ with a discrete co-compact closed subgroup $\Gamma$. Let $\mu$ be a finite invariant measure on  right cosets $\Gamma\backslash G$ \cite[Theorem 2.49]{gb1995course} that is normalized so that $\mu(\Gamma\backslash G)=1$. Then the \emph{quasi-regular} representation $R$ of $(G,\Gamma)$ is defined on $L^2(\gbg):=L^2(\Gamma\backslash G,d\mu)$ by
\begin{equation*}
\left(R(y)\phi\right)(x)=\phi(xy) 
\end{equation*}
with $y\in G$, $x\in\Gamma\backslash G$, and $\phi\in L^2(\Gamma\backslash G)$. The symbol $C_c(G)$ ($C_c^\infty(G)$) will denote continuous (smooth) compactly supported functions from $G$ into $\bbC$ and, when $E$ is open in $G$, the symbol $C_E^\infty(G)$ will denote the subset of $C_c^\infty(G)$ with support contained in $E$.

The archetypal example of this setting is when $G=\bbR$ and $\Gamma=\bbZ$, in which case $\Gamma\backslash G$ can be identified with $[-1/2, 1/2)$. For these particular groups, $R$ decomposes as
\begin{equation*}
R=\bigoplus_{j\in\bbZ}\pi_j 
\end{equation*}
where $\pi_j$ is the character on $G=\bbR$ defined by $\pi_j(x)=e^{2\pi ijx}$. By the Poisson summation formula, for  $f\in C_c^\infty(G)$,
\[
\left\Vert \sum_{k\in\bbZ}f(\mydot+k)\right\Vert _{L^2(-1/2,1/2)}^2=\sum_{j\in\bbZ}\left|\left\langle \pi_j, f \right\rangle\right|^2 .
\]
Restricting to $f$ supported on $E=(-1/2,1/2)$ gives
\[
\left\Vert f\right\Vert _{L^2(E)}^2=\sum_{j\in\bbZ}\left|\left\langle \pi_j, f \right\rangle\right|^2
\]
for all $f\in C_E^\infty(G)$. This equality can be extended to all of $L^2(E)$ by density, showing that the harmonic exponentials $\{\pi_j\}$ form a Parseval frame for $L^2(E)$.  While this is just the Plancherel theorem for Fourier series, the derivation here serves to illustrate the general case.

Let $G$ be a locally compact group, let $dx=dm(x)$ be Haar measure on $G$, let $E\subset G$ be open with compact closure and $m(E)>0$, and let $\{\pi_j\,|\,j\in\bbN\}$ be a set of representations of $G$ on separable Hilbert spaces $\{\fH_j\, |\, j\in\bbN\}$.  Further, assume that $\{\pi_j\,|\,j\in\bbN\}$ has the property that, for each $j$ and all $f\in L^2(E)$, the operator defined by $\pi_j(f)=\int_G f(x)\pi_j(x)\, dx$  is a  Hilbert-Schmidt class operator on $\fH_j$. Then $\{\pi_j\}$ will be called an \emph{OVF of representations} for $L^2(E)$ provided there exist $B\ge A>0$ such that 
\begin{equation}
\label{eq:avertvert}
A\left\Vert f\right\Vert _{L^{2}(E)}^{2}\le\sum\left\Vert \pi_j(f)\right\Vert _{\ths}^2\le B\left\Vert f\right\Vert _{L^2(E)}^2
\end{equation}
for all $f\in L^2(E)$.  In this expression and subsequently, $\left\Vert \mydot\right\Vert_\ths$ denotes the Hilbert-Schmidt norm.  In the case that each $\fH_j=\bbC$, each $\pi_j(x)$ for $x\in G$ and $j\in \bbN$ can be viewed either as a scalar or as an operator on $\bbC$, and $\left\Vert\pi_j(f)\right\Vert_{\HS} = \left|\left\langle  \pi_j, f \right\rangle\right|$ for all $f\in L^2(E)$.  In particular, when each $\pi_j$ in the preceding paragraph is regarded as a representation on $\bbC$, the inequality \eqref{eq:avertvert} holds with $A=B=1$.

In the more general setting, the appropriate replacement for $E=\mhh$ as it occurs in the $(\bbR,\bbZ)$ case, will be called a $(G,\Gamma)$ \emph{reproducing set}; i.e., a non-empty open set $E$ with compact closure having the property that $EE^{-1}$ is disjoint from every conjugate of $\Gamma-\{\bone_G\}$. Existence of such an $E$ is equivalent to existence of a non-empty open set $U\subset G$ such that $\cup_{g\in G}\,g^{-1}Ug$ intersects $\Gamma$ only in the point $\bone_{G}$. Further, if $R$ is as above, then $R$ decomposes discretely as
\begin{equation}
\label{eq:r-decomp-with-mults}
R=\bigoplus\pi_j,\qquad(\text{listed with multiplicities})
\end{equation}
where each $\pi_j$ is represented on a Hilbert space $\fH_j$ and has finite multiplicity \cite[Lemma 9.2.7]{dietmar2009principles}. Now and henceforth it is assumed that $dx$ is chosen so that for $f \in C_c(G)$
\begin{equation*} 
\label{eq:Fubini-group}
\int_G f(x)\, dx = \int_{\gbg} \sum_\gamma f(\gamma x)\, d\mu(x)
\end{equation*}
for $\mu$ as described above.  This is possible by \cite[Theorem 2.49]{gb1995course}, and in this case, the above equality also holds for all $f\in L^1(G)$ \cite{bourbaki1963integration}. Further, $R(f)$ will denote the operator  on $\htil=L^2(\gbg)$ obtained by integrating the representation $R$ against a function $f\in L^1(G)$; i.e., 
\begin{equation}
\label{eq:integrated-R}
R(f)=\int_G f(y)R(y)\,dy .
\end{equation}

In what follows, the algebra of trace-class operators on a separable Hilbert space $\fK$ will be denoted by $L^1(\fK)$ and the Hilbert space of Hilbert-Schmidt class operators on $\fK$ will denoted by $L^2(\fK)$.  Further, the trace of an operator $T\in L^1(\fK)$ will be denoted $\Tr(T)$, the Hilbert-Schmidt inner product of $S$ and $T \in L^2(\fK)$ will be denoted $\left\langle S, T \right\rangle_{L^2(\fK)} = \Tr(T^*S)$.
%, and the norm of $T\in L^2(\fK)$ will be denoted and $\left\Vert T \right\Vert_{L^2(\fK)}$.

With the necessary background and terminology established, the objective in the remainder of this section is to prove the following:
\begin{theorem}
\label{parseval}
Let $E$ be a $(G,\Gamma)$ reproducing set. Then the decomposition of $R$ into irreducible representations $\{\pi_j\}$, listed with multiplicities, implies that
$\{\pi_j\}$ forms a Parseval OVF for $L^2(E)$.
\end{theorem}

This begins with a preliminary result:
\begin{lemma}
\label{thm:hs_lem}
Let $E\subset G$ be non-empty and open with compact closure, and let $\fH=L^2(E)$. Then $\check R: f\mapsto R(f)$ is a bounded linear map from $\fH$ into $L^2(\htil)$.
\end{lemma}

\begin{proof}
It is shown in \cite{arthur2005introduction} that, for $f\in C_c(G)$, the Hilbert-Schmidt norm of $R(f)$ is given by
\begin{equation}
\label{eq:hs-norm-Rf-01}
\left\Vert R(f)\right\Vert_{\HS}^2=\int_{\Gamma\backslash G}\int_{\Gamma\backslash G}\left| \sum_{\gamma\in\Gamma}f(x^{-1}\gamma y)\right|^2\, d\mu(x)\, d\mu(y) .
\end{equation}
Because $E$ has compact closure, this formula carries over \emph{mutatis mutandis}
\begin{comment}
\textcolor{red}{I am not yet convinced of this.  I think I know how to show it, but it does not seem the same as the Arthur result, and will take a little more space.  (A page or so?)}
\end{comment}
to the case $f\in L^2(E)$. If $q:G\rightarrow\Gamma\backslash G$ is the canonical quotient map and $F$ is a measurable subset of $G$, then
\[
\int_G \chi_F(x)\,dx=\int_{\Gamma\backslash G} \sum_{\gamma\in \Gamma} \chi_F(\gamma x)\, d\mu( x) \ge \int_{\gbg} \chi_{q(F)}\, d\mu .
\]
By \cite[Lemma 2.46]{gb1995course}, there is a compact set $K\subset G$ such that $q(K)=F$.  Thus, given $S\subset \gbg$ and taking $F= q^{-1}(S)\cap K$ in the above yields $\int_K \chi_S \circ q(x)\, dx \ge \int_{\gbg} \chi_S\, d\mu$.  That is, $\int_K g\circ q(x)\, dx \ge \int_{\gbg} g\, d\mu$ for all characteristic functions $g$ on $\gbg$, and thus all non-negative measurable functions on $\gbg$. The right-hand side of (\ref{eq:hs-norm-Rf-01}) then becomes bounded
by
\[
\int_{\Gamma\backslash G}\int_{K}\left|\sum_{\gamma\in\Gamma}f(x^{-1}\gamma y)\right|^{2}\, dx\, d\mu(y)
\leq 
\int_K \int_K \left|\sum_{\gamma\in\Gamma}f(x^{-1}\gamma y)\right|^2\, dx\, dy .
\]
The sum in the integrand vanishes off the set $\Gamma_0=\Gamma\cap KEK^{-1}$, which is compact and discrete, hence finite. An application of the Cauchy-Schwarz inequality yields the following upper bound for  $\left\Vert R(f)\right\Vert _{L^2(\htil)}^2$,
\begin{align*}
\left\Vert R(f)\right\Vert _{\HS}^2&\leq  |\Gamma_0|\int_K \int_K \sum_{\gamma\in \Gamma_0}\left|f(x^{-1}\gamma y)\right|^2\, dx\, dy\\
 &\leq  |\Gamma_0|\sum_{\gamma\in \Gamma_0}\int_K \int_G \left|f(x^{-1}\gamma y)\right|^2\, dx\, dy\\
 & \le|\Gamma_0|^2 m(K)||f||_{L^2(E)}^2
\end{align*}
as desired.
\end{proof}

Now (\ref{eq:r-decomp-with-mults}) and (\ref{eq:integrated-R}) yield a unitary $V:\bigoplus \fH_{j}\rightarrow\htil$ for which, as an operator on $\bigoplus \fH_j$,
\[
V^{*}R(f)V=\bigoplus\pi_j(f).
\]
It follows that each $\pi_j(f)$ is a Hilbert-Schmidt class operator on $\fH_j$ and that
\begin{equation}
\label{eq:vertrf}
\left\Vert R(f)\right\Vert _{\HS}^2=\sum\left\Vert \pi_j(f)\right\Vert _{\HS}^2 . 
\end{equation}

The condition for the operators $\{\pi_j\}$, which respectively map into the Hilbert spaces $L^2(\fH_j)$, to form a Parseval OVF for $L^2(E)$ is 
\begin{equation}
\label{eq:parseval-g-frame-of-reps-gen}
\left\Vert f\right\Vert_{L^2(E)}^2=\sum\left\Vert \pi_j(f)\right\Vert_{\HS}^2 .
\end{equation}
In view of (\ref{eq:vertrf}), this inequality follows from $\left\Vert f\right\Vert_{L^2(E)}^2=\left\Vert R(f)\right\Vert _{\HS}^2$, a sufficient condition for which is that $E$ is a $(G,\Gamma)$ reproducing set. Verification of this sufficiency is achieved in the following lemmas.

\begin{lemma}
\label{thm:cod_hs}
Let $M\in L^1(\htil)$ and $E$ be an open subset of $G$ with compact closure and positive Haar measure. Then the function $f^M:E\rightarrow \bbC$ defined by $f^M(x) = \Tr(R(x^{-1})M)$ is bounded,
\begin{comment}
\textcolor{red}{and measurable?}
\end{comment}
 and $\check R^{*}M=f^M$.
\end{lemma}
\begin{proof}
First it will be shown that $f^M$ is well-defined. If $M$ has eigenvalues
$\{\lambda_j\, |\, j\in\bbN\}$ and corresponding eigenbasis $\{e_j\}\subset\htil$,
and $U$ is any unitary operator on $\htil$, then 
\begin{align*}
\left|\Tr(UM)\right| & \le\sum_j\left|\left\langle UMe_j,e_j\right\rangle_{\htil}\right| \\
& \le\sum_j\left\Vert UMe_j\right\Vert_{\htil} \\
& =\sum_j\left\Vert Me_j\right\Vert_{\htil}=\sum_j |\lambda_j| .
\end{align*}
Thus, $f^M(x)=\sum_j \left\langle R\left(x^{-1}\right)Me_j, e_j \right\rangle$ converges absolutely to a bounded
function on $E$.

It will now be shown that 
\[
\left<R(f),M\right>_{\ltht}=\left<f,f^{M}\right>_{L^{2}(E)} .
\]
The right-hand side is equal to
\[
\int_{E}f(x)\Tr(M^{*}R(x))\, dx .
\]
As implied by the above estimates, the series $\Tr(M^{*}R(x))=\overline{\Tr(R(x^{-1})M)}$, expanded using $\{e_j\}$, converges absolutely to a bounded function. This means the integrand is dominated by a multiple of $|f(x)|$ and, since $f\in L^2(E)\subset L^1(E)$, it follows from the dominated convergence theorem that
\begin{align*}
\left<f,f^M\right>_{L^2(E)} & =\Tr\left(\int_E f(x)M^* R(x)\, dx\right)
\end{align*}
which is just 
\[
\Tr\left(M^* \int_E f(x)R(x)\, dx\right) .
\]
The latter is equal to $\left\langle R(f),M\right\rangle _{\ltht}$, as desired.
\end{proof}

\begin{lemma}
\label{thm:parseval_lem}
Let $E$ be a $(G,\Gamma)$ reproducing set and $f\in L^2(E)$. Then \eqref{eq:parseval-g-frame-of-reps-gen} holds. 
\end{lemma}
 
\begin{proof}
Suppose $f\in C_E^\infty(G)$.  By \cite{arthur2005introduction}, $R(f)$ is trace-class.  Thus, with the notation $f_x(y)=f(yx)$, Lemma~\ref{thm:cod_hs} implies that the function $\check R^* R(f)$ has the following very specific form:
\begin{align}
(\check R^* R(f))(x) & = \Tr\left(R\left(x^{-1}\right)R(f)\right)\nonumber \\
& =\Tr\left(R\left(x^{-1}\right) \int_G f(y)R(y)\, dy \right)\nonumber \\
& = \Tr\left(\int_G f(y)R(y)\, dy R\left(x^{-1}\right)\right)\nonumber \\
  & =\Tr\left(\int_G f_x (y)R(y)\, dy\right)\nonumber \\
 & =\Tr\left(R(f_x)\right)\nonumber \\
 & =\int_{\Gamma\backslash G}\sum_{\gamma\in\Gamma}f_x(y^{-1}\gamma y)\, d\mu(y) \label{eq:geometric-trace-01-from-arthur}\\
  & =f_x(\bone_G)\mu(\Gamma\backslash G)+\int_{\Gamma\backslash G}\sum_{\bone_G\ne\gamma\in\Gamma}f_x(y^{-1}\gamma y)\, d\mu(y) \nonumber
\end{align}
where (\ref{eq:geometric-trace-01-from-arthur}) follows from the Selberg Trace Formula applied to the function $f_x$ (see \cite{arthur2005introduction}). If $x$ is such that $\supp f_x$ is disjoint from all conjugates of $\Gamma-\{\bone_G\}$, then the integral term vanishes and the right-hand side becomes $f_x\left(\bone_G\right)$, which is just $f(x)$. But this will happen if $x\in E$, since $\supp f_x=Ex^{-1}\subset EE^{-1}$, which has the desired disjointness property.

Hence, for $x\in E$ and $f\in C_E^{\infty}(G)$, $(\check R^* R(f))(x)=f(x)$. Consequently, $\left\Vert R(f)\right\Vert _{\HS}^2  =\left\langle \check R^* R(f), f \right\rangle_{\fH}  = \left\Vert f\right\Vert _{\fH}^2$
%\begin{align*}
%\left\Vert R(f)\right\Vert _{\HS}^2 & =\left\langle \check R^* R(f), f \right\rangle_{\ltht}\\
%&  = \left\Vert f\right\Vert _{\fH}^2
%\end{align*}
for all $f$ in a dense subspace of $\fH=L^2(E)$, and hence for all of $\fH$. As noted above, the desired Parseval frame condition (\ref{eq:parseval-g-frame-of-reps-gen}) follows from this equality. 
\end{proof}

This section has established that, if $G$ is a locally compact, unimodular Lie group, $\Gamma$ is a discrete, co-compact, closed subgroup, $\{\pi_j\, |\, j\in\bbN\}$ is a list (with multiplicities) of the subrepresentations of the quasi-regular representation of $(G,\Gamma)$, and $E$ is a $(G,\Gamma)$ reproducing set, then $\{\pi_j\}$ is an OVF of representations for $L^2(E)$ with $A=B=1$. Such entities will be called \emph{harmonic} OVFs.

\section{OVFs of representations for the real Heisenberg group}
\label{sec:heis}

With the context established in Section \ref{sec:Harmonic-g-frames-of-reps}, this section returns to the matter of finding a generalization of the Duffin-Schaeffer theorem on non-harmonic Fourier series in which the pair $(\bbR,\bbZ)$ is replaced by a more general $(G,\Gamma)$. In this setting, when one is given a reproducing neighborhood $E$ and a harmonic OVF of representations $\{\pi_j\, |\, j\in\bbN\}$ for $L^2(E)$, one may ask whether there is a ``neighborhood'' of $\{\pi_j\}$ consisting only of OVFs of representations for $L^2(E)$; i.e., consisting only of sequences of representations $\{\tilde{\pi}_j\, |\, j\in\bbN\}$ for which there are $B\ge A >0$ such that
\[
A\left\Vert f\right\Vert _{L^2(E)}^2\le\sum\left\Vert \tilde{\pi}_j(f)\right\Vert _{\ths}^2\le B\left\Vert f\right\Vert_{L^2(E)}^2
\]
for all $f\in L^2(E)$. This section takes up this question for $\bbH$, the real
Heisenberg group, defined as ordered triples $(x,\xi,t)\in\bbR^{\hn}\times\bbR^{\hn}\times\bbR$ with the operation
\[
(x,\xi,t)(x',\xi',t')=\left(x+x',\xi+\xi',t+t'+\frac{1}{2}(x\mycdot \xi'- x' \mycdot \xi)\right).
\]
The discrete subgroup $\Gamma$ consists of ordered triples in $\bbZ^{\hn}\times\bbZ^{\hn}\times\frac{1}{2}\bbZ$ and the reproducing neighborhood $E$ will be $D \times \mqqo$, where $D=\mhh^n \times \mhh^n$.

It is necessary to verify that $E$ really is a $(\bbH,\Gamma)$ reproducing
set. To see this, first observe that $\Gamma-\{0\}^{2n+1}= \Gamma_1 \cup \Gamma_2$ with $\Gamma_1 = \left(\bbZ^{2n}-\{0\}^{2n}\right)  \times \frac{1}{2}\bbZ$ and $\Gamma_2 = \{0\}^{2n} \times \left(\frac{1}{2}\bbZ - \{0\}\right)$.  Since the first $2n$ scalar components of $EE^{-1}$ lie in $(-1,1)$ and since the orbit of $\Gamma_1$ under conjugation in $G$ consists only of members of $(\bbZ^{2n}-\{0\}^{2n})\times\bbR$, $EE^{-1}$ is disjoint from this orbit.
%
%To see this, first observe that the first $2n$ components of $EE^{-1}$ are the open intervals $(-1, 1)$.   Any conjugate of $\Gamma_1 = \left(\bbZ^{2n}-\{0\}^{2n}\right)  \times \bbZ \subset \Gamma$, has a nonzero integer in one of its first $2n$ scalar components, and thus does not intersect $EE^{-1}$.  
On the other hand, $\Gamma_2$ 
%$= \{0\}^{2n} \times \left(\bbZ - \{0\}\right) \subset \Gamma$ 
is in the center of $\bbH$, so it is equal to its orbit under conjugation.  If $(x,\xi,t)\in \bbH$, then $(x, \xi, t)^{-1}=(-x,-\xi,-t)$, so if $(x,\xi,t),(x',\xi',t')\in E$ and if $(x,\xi,t)(x', \xi', t')^{-1}\in \Gamma_2$, then $x=x'$, $\xi=\xi'$, and $t-t'\in \frac{1}{2}\bbZ-\{0\}$, which is impossible since $t,t'\in\mqqo$.  Thus,  $EE^{-1}$ does not intersect $\Gamma_2$.

It remains to explicitly describe the subrepresentations of $R$ and their corresponding multiplicities.  Up to equivalence, the representations of $\bbH$ are of two types. The infinite-dimensional representations of $\bbH$ have the form (see \cite{thangavelu2009harmonic}) $\rho_\omega:\bbH\times L^2(\bbR^{\hn})\rightarrow L^2(\bbR^{\hn})$ 
\begin{equation*}
%\label{eq:infinite-dim-reps-heisen}
\left(\rho_{\omega}(x,\xi,t)\phi\right)(y)=e^{-2\pi i\omega(t+x\mycdot y + \frac{1}{2}x\mycdot \xi)}\phi(y+x)
\end{equation*}
with $\omega \in\bbR^*=\bbR-\{0\}$ and $\phi\in L^2(\bbR^{\hn})$.
The others are (one-dimensional) characters, given by $\chi_{b,\beta}(x,\xi,t)=e^{-2\pi i(b\mycdot x+\beta\mycdot\xi)}$ for $b,\beta \in \bbR^{\hn}$. 
To decompose $L^2(\Gamma\backslash H_n)$ into $R$-invariant subspaces, observe first that $g\in L^2(\Gamma\backslash H_n)$ may be viewed as a function on $H_n$ that is invariant under left translations in $\Gamma$.  Such a function satisfies, in particular, $g(x,\xi,t)=g(x,\xi,t+1/2)$.  Thus, $L^2(\Gamma\backslash H_n) = \bigoplus_{k\in\bbZ} \fK_{2k}$, where $\fK_{2k}$ is the $R$-invariant space $\{h\in L^2(\Gamma\backslash H_n): h(x,\xi,t) = e^{4\pi i k t} h(x,\xi,0)\}$.  The action of $R$ on $\fK_0$ factors through the action of the right regular representation of $\bbR^{2n}$ on $L^2(\mathbb{T}^{2n})$, and can therefore be shown to decompose into the sum
\[
\bigoplus_{a,\alpha\in\bbZ^n} \chi_{a,\alpha} .
\]
Further, it is shown in \cite{thangavelu2009harmonic} that the action of $R$ on $\fK_{2k}$, $k\ne 0$, splits into $|2k|^n$ irreducible actions, each of which is equivalent by a Weil-Brezin-Zak transform to the action of $\rho_{2k}$ on $L^2(\bbR^n)$.  Thus,
\[
R\cong\bigoplus_{a,\alpha\in\bbZ^{\hn}}\chi_{a,\alpha}\oplus\bigoplus_{k\in\bbZ^{*}}{ |2k|}^{\hn}{ \rho_{2k}}
\]
where $\bbZ^* = \bbZ-\{0\}$.  From this it follows that the frame condition (\ref{eq:parseval-g-frame-of-reps-gen}) becomes
\begin{equation*} %\label{eq:harmonic-heisenberg}
\left\Vert f\right\Vert _{L^2(E)}^2=\sum_{a,\alpha\in\bbZ^{\hn}}\left|\chi_{a,\alpha}(f)\right|^2+\sum_{k\ne 0}{ |2k|}^{\hn}\left\Vert { \rho_{2k}}(f)\right\Vert_\HS^2
\end{equation*}
for all $f\in L^2(E)$. The goal of this paper can now be rephrased as the following result about perturbing the values of the equispaced parameters $a$, $\alpha$, and ${ 2k}$ to vectors $\{b_a\, |\, a\in\bbZ^{\hn}\}\subset \bbR^{\hn}$ and $\{\beta_\alpha\, |\,\alpha\in\bbZ^{\hn}\}\subset \bbR^{\hn}$ and real numbers  $\{\omega_k\, |\, k\in\bbZ^*\}$.

\begin{theorem}
\label{thm:non-harmonic-heisenberg}
Suppose $\{b_a\, |\, a\in\bbZ^{\hn}\}$ and  $\{\beta_\alpha\, |\, \alpha\in\bbZ^{\hn}\}$ are sequences of real $n$-vectors and $\{\omega_k \,|\,k\in \bbZ^*\}$ is a sequence of  real numbers.  Define 
\[
M = \max\left\{\sup_{a\in\bbZ^{\hn}} \left\Vert b_a - a\right\Vert_{\infty} , \sup_{\alpha\in\bbZ^{\hn}}\left\Vert \beta_\alpha - \alpha\right\Vert_{\infty}, \sup_{k\ne 0} \left| \omega_k-{ 2k} \right|  \right\} .
\]
If $M>0$ is sufficiently small, then there exist $A=A(M)>0$ and $B=B(M)$ such that
\begin{equation*}
A\left\Vert f\right\Vert _{L^2(E)}^2
\le \sum_{a,\alpha}\left|\chi_{b_a,\beta_\alpha}(f)\right|^2+\sum_{k\ne 0}{ |2k|}^{\hn}\left\Vert \rho_{\omega_k}(f)\right\Vert_\HS^2
\le B\left\Vert f\right\Vert _{L^2(E)}^2
\end{equation*}
holds for all $f\in L^2(E)$.
\end{theorem}
\begin{proof}
Let $f\in L^2(E)$.  For $b,\beta, \omega \in\bbR^{\hn}$, the (Euclidean) Fourier transform of $f$ at $(b,\beta,\omega)$ is defined to be
\[
\hat{f}(b,\beta,\omega)=\int\int\int f(x,\xi,t)e^{-2\pi i(b\mycdot x+\beta\mycdot\xi+\omega t)}\, dx\, d\xi\, dt .
\]
Let $\mathcal{F}_1$, $\mathcal{F}_2$, and $\mathcal{F}_3$ denote the corresponding Fourier transforms with respect to the first, second, and third variables, respectively. Further, the symbols $p$, $q$, and $r$ will denote the quadratic forms 
\[
q(f)=\sum_{a,\alpha}\left|\chi_{b_a,\beta_\alpha}(f)\right|^2
\]
and
\[
r(f)=\sum_{k\ne0}{ |2k|}^{\hn}\left\Vert \rho_{\omega_{k}}(f)\right\Vert _{\HS}^{2}
\]
and
\[
p(f)=q(f)+r(f).
\]
The result to be proven, in effect, is that for $M>0$ sufficiently small, the seminorm $p^{1/2}$ is equivalent to $\left\Vert \mydot \right\Vert_{L^2(E)}$.

The key step in this proof will be a simple extension of Duffin and Schaeffer's \cite[Lemma II]{duffin1952class} for the domain $E$. Specifically, given $J=\dsintparams$ and given $\tilde{\ }: J \rightarrow \bbR^{2\hn+1}$ and given that the number 
\[
M' = \sup_{z\in J} \left\Vert \tilde{z} - z \right\Vert_{\infty}
\]
is sufficiently small, there is $T=T(M')$ such that
\[
\sum_{z\in J}\left|\hat{f}(\tilde{z})-\hat{f}(z)\right|^2 \le T(M')\sum_{z \in J}\left|\hat{f}(z)\right|^2
\]
for every $f\in L^{2}(E)$. By the triangle inequality, this means that the quantity
\begin{equation}
\sum_{z\in J}\left|\hat{f}(\tilde{z})\right|^2
\label{eq:nonharmic-fourier-series}
\end{equation}
is bounded above and below by the quantity
\begin{equation*} %\label{eq:nonharmonic-fourier-series}
\left(1\pm T(M')^{1/2}\right)^2\sum_{z\in J}\left|\hat{f}(z)\right|^2=\left(1\pm T(M')^{1/2}\right)^2\left\Vert f\right\Vert _{L^2(E)}^2 .
\end{equation*}
Thus, it suffices to show that $p(f)$ is bounded above and below
by positive multiples of \eqref{eq:nonharmic-fourier-series} for some
$\tilde{z}$'s for which $M'=M$. 

For $\omega\ne 0$ and $f\in L^2(E)$, it will be useful to obtain a formula for $\left\Vert \rho_{\omega}(f)\right\Vert _{\HS}^{2}$.  By an argument in Chapter 7 of \cite{gb1995course}, the operator $\rho_{\omega}(f):L^{2}(\bbR^{\hn})\rightarrow L^2(\bbR^{\hn})$ has Hilbert-Schmidt norm
\[
\left\Vert \rho_{\omega}(f)\right\Vert _{\HS}^2=\frac{1}{|\omega|^{\hn}}\int\int\left|\mathcal{F}_3 f(u,v,\omega)\right|^2\, du\, dv
\]
for Haar measure on $H_n$ normalized to coincide with Lebesgue measure on $\bbR^{2\hn + 1}$. Further, the facts that $g=\mathcal{F}_3f(\mydot,\mydot,\omega)$ is supported on $D$ and is square-integrable imply that $\left\Vert \rho_{\omega}(f)\right\Vert _{\HS}^2$ may be written using the $2n$-dimensional Fourier series expansion of $g$ as
\[
 \left\Vert \rho_{\omega}(f)\right\Vert _{\HS}^{2} = \frac{1}{|\omega|^{\hn}} \sum_{a,\alpha\in\bbZ^{\hn}}\left|\mathcal{F}_1\mathcal{F}_2\mathcal{F}_3f(a,\alpha,\omega)\right|^2=\frac{1}{|\omega|^{\hn}} \sum_{a,\alpha}\left|\hat{f}(a,\alpha,\omega)\right|^2 
\]
for any $\omega\ne 0$.

Consider $|r(f)-\myphi(f)|$, where 
\begin{align}
\myphi(f) & =\sum_{k\ne 0}\sum_{a,\alpha}\left|\hat{f}(a,\alpha,\omega_k)\right|^2
\label{eq:expanded-r} \\
& = \sum_{k\ne 0} \myphikf . \nonumber 
\end{align}
The quantity has the following upper bound:
\begin{align*}
\left|r(f)-\myphi(f)\right| 
& \le\sum_{k\ne0}\left|\left|\frac{{ 2k}}{\omega_{k}}\right|^{\hn}-1\right|\myphikf\\
 & \le\sup_{k\ne0}\left|\left|\frac{{ 2k}}{\omega_{k}}\right|^{\hn}-1\right|\sum_{k\ne0}\myphikf\\
 & =\sup_{k\ne0}\left|\left|\frac{{ 2k}}{\omega_{k}}\right|^{\hn}-1\right| \phi(f) .
\end{align*}
For $M \ll 1$, a bound may be obtained by replacing $\left|\omega_{k}\right|^{\hn}$
by ${ |2k|}^{\hn}-\hn M { |2k|}^{n-1}$. A corresponding bound
on the supremum terms is $\hn M/({ |2k|}-\hn M)$, which is decreasing
in $\left|k\right|$. Thus, the supremum term is less than $C(M)=\hn M/({ 2}-\hn M)$,
which goes to zero as $M$ goes to zero. In other words, 
\begin{equation} \label{eq:inf-dim-fourier-terms-ineq}
(1-C(M))\myphi(f)\le r(f)\le(1+C(M))\myphi(f) .
\end{equation}

The inequality
\begin{equation} 
\label{eq:simplyobtainedineq}
(1-C(M))(\phi(f)+q(f))\le p(f)\le(1+C(M))(\phi(f)+q(f))
\end{equation}
results from adding $(1-C(M))q(f)\le q(f)\le(1+C(M))q(f)$ to \eqref{eq:inf-dim-fourier-terms-ineq}. For each $b,\beta\in\bbR^{\hn}$, the quantity $\chi_{b,\beta}(f)$ is equal to $\hat{f}(b,\beta,0)$,
so
\begin{equation*}
q(f)=\sum_{a,\alpha\in\bbZ^{\hn}}\left|\hat{f}(b_a,\beta_\alpha,0)\right|^2 %\label{eq:expanded-q}
\end{equation*}
for Haar measure as above. Thus, combining the above with \eqref{eq:simplyobtainedineq} and \eqref{eq:expanded-r} gives
\[
(1-C(M))\sum_{z\in J}\left|\hat{f}(\tilde{z})\right|^2
\le p(f)
\le(1+C(M))\sum_{z\in J}\left|\hat{f}(\tilde{z})\right|^2
\]
where, when $k=0$, $(a,\alpha,{ 2}k)\tilde{\ }=(b_a, \beta_\alpha, 0)$ and, when $k\ne 0$, $(a,\alpha,{ 2}k)\tilde{\ }=(a,\alpha, \omega_k)$. For these values of $\tilde{z}$, the number $M'$ is equal to $M$, and
\begin{align*}
& (1-C(M))(1-T(M)^{1/2})^2\left\Vert f\right\Vert _{L^{2}(E)}^{2} \\
& \le p(f) \\
& \le(1+C(M))(1+T(M)^{1/2})^2\left\Vert f\right\Vert _{L^{2}(E)}^{2}
\end{align*}
as desired.
\end{proof}
 
Observe that by making the perturbations small, $A$ and $B$ can be made as close to one as desired, resulting in a ``nearly Parseval'' OVF of representations. Thus, viewing the list of representations $\{\chi_{b_a,\beta_\alpha}\}\cup\{\rho_{\omega_k}\}$ with the appropriate number of repetitions, the desired result about OVFs of representations on $\bbH$ is obtained:  all that is needed to specify one is a sequence of numbers satisfying a Duffin-Schaeffer type stability condition. In particular, since an element $f$ in a Hilbert space $\fH$ is uniquely specified by $\{T_j f\}$ when $\{T_j\}$ is a OVF for $\fH$, the preceding shows that $f\in L^2(E)$ is uniquely specified by $\{\pi_j(f)\}$.

\section{Conclusion}

The preceding sections have described what it means for a OVF of representations on a locally compact Lie group $G$ to be harmonic. For $G=\bbH$, $\Gamma=\bbZ^{2n}\times \frac{1}{2}\bbZ$, and $E=D\times \mqqo$, a family of OVFs that are are ``almost harmonic'' was constructed by perturbing a harmonic OVF of representations in a particular way.  This construction is analogous to the development of frames of non-harmonic exponentials in $L^2(E)$ starting with an orthonormal basis of harmonic exponentials given by Duffin and Schaeffer. The OVFs constructed here appear to stand in contrast those found in current literature, which are typically generated as the unitary orbit of a single fixed operator \cite{Han:2014kx,kaftal2009operator}.
The nature of the OVFs introduced here is more similar to the non-harmonic Fourier frames of \cite{duffin1952class} than to wavelet or Gabor systems, and they are more closely related to the problem in representation theory described above.

The development in this paper is restricted to OVFs for $L^2(E)$, where $E$ is a proper subset of $G$. A possible extension of interest is the case where $E=G$, seeking a theory that provides an analysis of $L^2(G)$ that provides features akin to Gabor analysis for $L^2(\bbR)$.

As noted, the condition set forth by Duffin and Schaeffer to get a Fourier frame is quite general, whereas the condition given in this paper is less so. It would be interesting to quantify the deviation from harmonic OVFs that is possible while still remaining within the set of OVFs.

\nocite{*}
\bibliographystyle{spmpsci}      
\bibliography{OVFbib_abbrev}

\end{document}